\documentclass[11pt,reqno]{amsart}
\usepackage{amssymb,latexsym}
\usepackage{graphicx}
\usepackage{url}		
\usepackage{fancyhdr}
\calclayout

\makeatletter
\newcommand{\monthyear}[1]{%
  \def\@monthyear{\uppercase{#1}}}
\newcommand{\volnumber}[1]{%
  \def\@volnumber{\uppercase{#1}}}
\AtBeginDocument{%
\def\ps@plain{\ps@empty
  \def\@oddfoot{\@monthyear \hfil \thepage}%
  \def\@evenfoot{\thepage \hfil \@volnumber}}
\def\ps@firstpage{\ps@plain}
\def\ps@headings{\ps@empty
  \def\@evenhead{%
    \setTrue{runhead}%
    \def\thanks{\protect\thanks@warning}%
    \uppercase{Primality testing and integer factorization}\hfil}%
  \def\@oddhead{%
    \setTrue{runhead}%
    \def\thanks{\protect\thanks@warning}%
    \hfill\uppercase{Primality testing and integer factorization}}%
  \let\@mkboth\markboth
  \def\@evenfoot{%
    \thepage \hfil \@volnumber}%
  \def\@oddfoot{%
    \@monthyear \hfil \thepage}%
  }%
\footskip=25pt
\pagestyle{headings}%
}
\makeatother

\theoremstyle{plain}
\numberwithin{equation}{section}
\newtheorem{thm}{Theorem}[section]
\newtheorem{Theorem}[thm]{Theorem}
\newtheorem{Lemma}[thm]{Lemma}
\newtheorem{example}[thm]{Example}
\newtheorem{definition}[thm]{Definition}

\newtheorem{corollary}[thm]{Corollary}

\begin{document}
\fancyhf{}
\chead{\textbf{On $\pi(n)$ and $\sum_{3 \leq p \leq n} \frac{1}{p}$}}
\cfoot{\textbf{}\ \thepage}
\pagestyle{fancy}

\monthyear{12 2016}
\volnumber{Volume, Number}
\setcounter{page}{1}

\title{On the prime counting function and the partial sum of reciprocals of odd primes}
\author{Madieyna Diouf}
\email{mdiouf1@asu.edu}
\markright{Onto $\pi(n)$ and $\sum_{3 \leq p \leq n} \frac{1}{p}$}
\begin{abstract}
We present a function that tests for primality, factorizes composites and builds a closed form expression of $\pi(n^2)$ in terms of $\sum_{3 \leq p \leq n} \frac{1}{p}$ and a weaker version of $\omega(n)$. 
\end{abstract}
\maketitle
\section{Introduction}
The prime counting function $\pi({n})$, the partial sum of the reciprocals of odd primes  $\sum_{3 \leq p \leq n} \frac{1}{p}$, and $\omega(n)$ the number of distinct prime factors of $n$, have historic interests for being well studied by Euler\cite{bibiloni2006series}, Hardy, Ramanujan\cite{hardy2000normal}, and Erdos\cite{erdos1935normal}. We establish an equation that unites these three functions. \\
\: \: Consider a function $f(n,x)$ that takes two positive odd integers $n$ and $x$ and returns the smallest odd multiple of $x$ exceeding $n$. Example: $f(1, 3) = 3$, $f(5, 3) = 9$. We observe that, \\
\: \:  the minimum of $f(n,p)$ taken over odd primes $p\leq \lceil \sqrt{n} \: \rceil$, is the smallest odd composite exceeding $n$; call this $c_1$.\\
\: \:  If $n$ is odd, then $n+2$ is a prime if and only if the gap between $c_1$ and $n$ is $4$ or $6$. Hence, one has a primality test based on computing $f(n,x)$. We also point out that,\\
\: \: if $n$ and $x$ are positive odd integers, then $x\: | \: n$ if and only if $f(n-2, x) = n$; this gives a second primality test and factorization criterion. \\
\: \: Furthermore, $f(n,x)$ is used to count the number of odd integers in the interval $(n^2, (n+2)^2]$. This count is set to be equal to the expected number of odd integers in the same interval given by $\frac{(n+2)^2 - n^2}{2}$. The resulting equality gives an exact formula of $\sum_{3 \leq p \leq n} \frac{1}{p}$ and the outcome in this argument is illustrated by examples and applications.\\
\: \: We exhibit a closed form expression 
of $\pi({n^2})$ in terms of $\sum_{3 \leq p \leq n} \frac{1}{p}$, and a weaker version of $\omega(n)$, and compare the expression of $\pi(n)$ provided by the prime number Theorem with $\pi(n^2)$ that rises from counting the odd integers in the interval $(1, n^2]$. 
\section{Introducing the function $f(n,x)$}
\begin{definition}
\textup {
\\ Given a positive odd integer $n > 3$, let $A_n$ denote the set of all odd primes less than or equal to the ceiling of square root of $n$. That is, $A_{n}= \{3 \ldots \lceil \sqrt{n} \: \rceil \}$ Define a function 
}
\begin{equation}
f(n, x) = n + 2x - (n - x) mod(2x),  \:\: where \: x  \in A_{n}.
\end{equation}
For each prime  $x$ in $A_n$, the function $f(n, x)$ yields the smallest odd integer multiple of $x$ greater than $n$. \\
\end{definition}
\begin{example}
let $n = 81$, we have $ \lceil \sqrt{n} \:  \rceil \ =9$. Hence, $A_{n}=\{3.5.7\}$.  
\[ f(81, 3) = 81 + 2(3) - [(81 - 3)mod(2*3)] =87 .\]
\[ f(81, 5) = 81+ 2(5) - [(81 - 5)mod(2*5)] =85 .\]
\[ f(81, 7) = 81 + 2(7) - [(81- 7)mod(2*7)] =91 .\]
$f(81, 3) = 87$, is the smallest odd multiple of $3$ greater than $81$.\\
$f(81, 5) = 85$, is the smallest odd multiple of $5$ greater than $81$.\\
$f(81, 7) = 91$, is the smallest odd multiple of $7$ greater than $81$.\\
\end{example}
\begin{example}
For $n = 111$, the set of all odd primes less than or equal to $\lceil \sqrt{111} \:  \rceil$ is $A_{n} =\{3, 5, 7, 11\}.$ Note: The prime number $11$ is in the set $A_n$ because of the ceiling of square root of $n$. \\ f(n, x) = n + 2x - [(n - x) mod(2x)] where $x$ is in $A_{n}.$
\[ f(111, 3) = 111 + 2(3) - [(111 - 3)mod(2*3)] =117 .\]
\[ f(111, 5) = 111 + 2(5) - [(111 - 5)mod(2*5)] =115 .\]
\[ f(111, 7) = 111 + 2(7) - [(111 - 7)mod(2*7)] =119 .\]
\[ f(111, 11) = 111 + 2(11) - [(111 - 11)mod(2*11)] =121 .\]
$f(111, 3) = 117$, is the smallest odd multiple of $3$ greater than $111$.\\
$f(111, 5) = 115$, is the smallest odd multiple of $5$ greater than $111$.\\
$f(111, 7) = 119$, is the smallest odd multiple of $7$ greater than $111$.
\end{example}
\textbf{What the function $f(n, x)$ does? } The value $(n - x) mod(2x)$ is the distance between $n$, to the largest odd multiple of $x$, less than or equal to $n$; therefore, $ n - [(n - x) mod(2x)]$ is the largest odd multiple of $x$ less than or equal to $n$. Add the distance $2x$ to the largest odd multiple of $x$ less than or equal to $n$, this gives the smallest odd multiple of $x$ that is greater than $n$. \\
\begin{Lemma} 
Given two positive odd integers $n$ and $x$, the function $f(n, x)$ yields the smallest odd multiple of $x$ exceeding $n$.\\
\end{Lemma}
\begin{proof}
(By contradiction), choose two positive odd integers $n$ and $x$, and suppose that $f(n, x)$ is not the smallest odd multiple of $x$ that is greater than $n$. This means that there exists an integer $k_{2}$ such that $ n < k_{2}x < f(n, x)$, that is,
\[\textup{(2)\: \: \: \: \: \: \: \: \: \:  \: \: \: }  k_{2}x <  n + 2x - [(n - x) mod(2x)].\textup{ \: \: \: \: \: \: \:  \: \: \: \:  \: \: }\] But since $n -  [(n - x) mod(2x)]$ is the largest odd multiple of $x$ less than or equal to $n$, there exists an integer $k_1$ such that $n -  [(n - x) mod(2x)] = k_1x.$ Thus, inequality $(2)$ becomes $k_{2}x  <   2x  +  k_{1}x$, which implies that $k_{2} -  k_{1} < 2$. This is impossible because ($k_2x$ is an odd multiple of $x$ that is) $ > n$, and ($k_1x$ is an odd multiple of $x$ that is) $\leq n$, so the value $k_2 - k_1$ cannot be less than $2$. Therefore, given two positive odd integers $n$ and $x$, the function $f(n, x)$ gives rise to the smallest odd multiple of $x$ exceeding $n$.
 \end{proof}

\textbf{Note:} Lemma $2.4$ is proven for any two arbitrary positive odd integers $n$ and $x$, but we will often let $x$ to be a prime in the interval $[3, \lceil \sqrt{n}\textup{ }\rceil]$.

\subsection{First primality test algorithm}
\textup{\\}
Given a positive odd integer $n > 3$, Let $p_n$ denote the largest prime less than or equal to $\lceil \sqrt{n}\textup{ }\rceil$. We build a list $L$ that contains the elements of $f(n,p)$ where $p$ is a prime limited to $3 \leq p \leq p_n$. Thus, $L = \{f(n, 3), f(n, 5), ... ,f(n, p_n)\}$, in other words, the elements of $L$ are, \{the smallest odd multiple of 3 exceeding $n$, the smallest odd multiple of 5 exceeding $n$, .... the smallest odd multiple of $p_n$ exceeding $n$\} not necessarily in this particular order. When $L$ is sorted in an increasing order, the first element in the sorted list is the minimum of the odd composites exceeding $n$. Let $c_1$ denote this composite, \textbf{we say that $c_1$ is the smallest odd composite exceeding $n$.} This means that if there is any odd integer $k$  between $n$ and $c_1$, then $k$ cannot be a composite. Therefore, $k$ must be a prime.\\\\
In example $1$ page $3$, the sorted list  $L = \{85, 87, 91\}$, this means $c_1 = 85$  is the smallest odd composite that is greater than $n = 81$. Now since there is a gap between $n$ and $c_1$ wide enough to hold one odd integer, we say that the odd integer in this gap that is $n+2 = 81 + 2 = 83 $ cannot be a composite; hence, it is a prime.

\begin{Theorem}
If $n$ is a positive odd integer, then $n + 2$ is a prime if and only if the gap between $c_1$ and  $n$ is $4$ or $6$.
\end{Theorem}

\begin{proof}
Choose a positive odd integer $n$, and suppose that $n + 2$ is a prime. The integer $c_1$ is the smallest odd composite greater than $n$. This means that $n - c_1$ must be equal to at least $4$; thus,  $n - c_1$ = $4$ or $6$. (This gap between $n$ and $c_1$ cannot be more than $6$ as we explained in a Note right below this proof on page 3). \\
Conversely, assume that $c_1 - n = 4$ or $6$. Since  $c_1$ is the smallest odd composite greater than $n$, we can state the following. If $c_1 - n = 4$, then the positive odd integer $n+2$ that is between $n$ and $c_1$ is not a composite; thus, $n+2$ is a prime. Similarly, if $c_1 - n = 6$, then the two positive odd integers $n+2$ and $n+4$ that are between $n$ and $c_1$ are not composites. Hence, $n+2$ and $n+4$ must be a pair of twin primes. This means that in whichever circumstance the case applies, we have $n+2$ is a prime. We conclude that $n + 2$ is a prime if and only if the gap between $c_1$ and  $n$ is either $4$ or $6$.
\end{proof}
\textbf{Note:} The gap between $n$ and $c_1$ cannot be greater than $6$ since it is well known that one number must be a multiple of $3$ in every sequence of three consecutive odd numbers.
\begin{corollary}
Given a positive odd integer $n$, we can prove that $n + 2$ is a lower member of a pair of twin primes if and only if the gap between $n$ and $c_1$ is $6$.
\end{corollary}

\begin{proof}
Assume that the gap between $n$ and $c_1$ is $6$, this implies that there are two positive odd integers between $n$ and $c_1$ that are not composite numbers. Hence, they are a pair of twin primes and $n+2$ is the lower member of the pair.\\ Conversely, suppose that $n+2$ is a lower member of a pair of twin primes, this means that $n+4$ is also a prime. Thus, the gap between $n$ and $c_1$ is 6. We conclude that given a positive odd integer $n$, we have $n + 2$ is a lower member of a pair of twin primes if and only if the gap between $n$ and $c_1$ is $6$.
\end{proof}

\begin{corollary}
Given a positive odd integer $n$, we can prove that $n + 2$ is a composite if and only if the gap between $n$ and $c_1$ is 2.
\end{corollary}

\begin{proof}
Assume that $n+2$ is a composite, so $n+2$ is the smallest odd composite that is greater than $n$, but this title belongs to $c_1$. Hence, $n + 2 = c_1$. This means that the gap between 
$n$ and $c_1$ is $2$. Conversely, suppose that the gap between $n$ and $c_1$ is $2$. This means that $n + 2 = c_1$ which is a composite. Therefore,  $n + 2$ is a composite if and only if the gap between $n$ and $c_1$ is $2$.
\end{proof}

\begin{example}
For $n = 111$, the set of all odd primes less than or equal to $\lceil \sqrt{111} \:  \rceil$ is $A_{n} =\{3, 5, 7, 11\}.\textup{ } f(n, x) = n + 2x - [(n - x) mod(2x)] \textup{ } where \textup{ } $x$ \textup{ } is \textup{ } in \textup{ }A_{n}.$
\[ f(111, 3) = 111 + 2(3) - [(111 - 3)mod(2*3)] =117 .\]
\[ f(111, 5) = 111 + 2(5) - [(111 - 5)mod(2*5)] =115 .\]
\[ f(111, 7) = 111 + 2(7) - [(111 - 7)mod(2*7)] =119 .\]
\[ f(111, 11) = 111 + 2(11) - [(111 - 11)mod(2*11)] =121 .\]
The sorted list $L =  \{115, 117, 119, 121\}$, so $c_1 = 115$. This means that $c_1 - n = 115 - 111 =  4$; thus, $n + 2 = 113$ is a prime by Theorem $2.5$. 
\end{example}

\begin{example}
For $n = 189$, the set of all odd primes less than or equal to $\lceil \sqrt{189} \:  \rceil$ is $A_{n} =\{3, 5, 7, 11, 13\}. \textup{ } f(n, x) = n + 2x - [(n - x) mod(2x)]  \textup{ } where \textup{ } $x$ \textup{ } is \textup{ } in \textup{ }A_{n}.$ 
\[ f(189, 3) = 189 + 2(3) - [(189 - 3)mod(2*3)] =195 .\]
\[ f(189, 5) = 189 + 2(5) - [(189 - 5)mod(2*5)] =195 .\]
\[ f(189, 7) = 189 + 2(7) - [(189 - 7)mod(2*7)] =203 .\]
\[ f(189, 11) = 189 + 2(11) - [(189 - 11)mod(2*11)] =209 .\]
\[ f(189, 13) = 189 + 2(13) - [(189 - 13)mod(2*13)] =195 .\]
The sorted list $L =  \{195, 195, 195, 203, 209\}$, so that $c_1 = 195$. Thus, $c_1 - n = 195 - 189 =  6$; therefore, $n + 2 = 191$ is a lower member of a pair of twin primes by Corollary $2.6$. 
\end{example}

\begin{example}
For $n = 297$, the set of all odd primes less than or equal to $\lceil \sqrt{297} \: \rceil$ is $A_{n} =\{3, 5, 7, 11, 13, 17\}. \textup{ }  f(n, x) = n + 2x - [(n - x) mod(2x)]  \textup{ } where \textup{ } $x$ \textup{ } is \textup{ } in \textup{ } A_{n}.$
\[ f(297, 3) = 297 + 2(3) - [(297 - 3)mod(2*3)] =303 .\]
\[ f(297, 5) = 297 + 2(5) - [(297 - 5)mod(2*5)] =305 .\]
\[ f(297, 7) = 297 + 2(7) - [(297 - 7)mod(2*7)] =301 .\]
\[ f(297, 11) = 297 + 2(11) - [(297 - 11)mod(2*11)] =319 .\]
\[ f(297, 13) = 297 + 2(13) - [(297 - 13)mod(2*13)] =299 .\]
\[ f(297, 17) = 297 + 2(17) - [(297 - 117)mod(2*17)] =323 .\]
The sorted list  $L =  \{299, 301, 303, 305, 319, 323\}$, so $c_1 = 299$. Hence, $c_1 - n = 299 - 297 =  2$; therefore, $n + 2 = 299$  is a composite by Corollary $2.7$. \\
\end{example}
\section{Primality testing and integer factorization \\using $f(n,x)$ as a standalone function}
\subsection{$f(n,x)$ is an integer factorization function.}
\begin{Theorem}
Given two positive odd integers $n$ and $x$, we can prove that $x$ is a factor of $n$ if and only if $x$ satisfies the equation $ f(n-2, x) = n$. 
\end{Theorem}
\begin{proof}
Let $n$ and $x$ be two positive odd integers and suppose that $x$ is a factor of $n$; thus, $n$ is the smallest odd multiple of $x$ that is greater than $n-2$. By Lemma $2.4$ on page $2$, this means that $f(n-2, x) = n. $ \\Conversely, assume that $f(n-2, x) = n$. By Lemma $2.4$, $f(n-2, x)$ gives the smallest odd multiple of $x$ greater than $n - 2$. This means that $f(n-2, x) = kx$  for some integer $k$. From our hypothesis we have $f(n-2, x) = n$, and now we have $f(n-2, x) = kx$; these two equations imply that $kx = n$, in other words $x$ is a factor of $n$. Therefore, given two positive odd integers $n$ and $x$, it holds true that $x$ is a factor of $n$ if and only if $x$ satisfies the equation $ f(n-2, x) = n.$
\end{proof}
\begin{example}
Using "Walfram, mathematica" \\
What are the prime factors of $n = 15015 = 3*5*7*11*13$? \\
$\textbf{Solution:}$ Given $n = 15015$, 
Theorem $3.1$ in subsection $3.1$ states that the integer solutions of the equation $f(n -2, x) = n$ are the factors of $n$. But since we know that the factors of $n$ are less than or equal to $n$ therefore, 
we can look for $x$ in the interval between $1$ and $n$ this is $1 \leq x \leq n$. Enter the following equation:
\begin{center}
$15013+2x - [(15013 - x) mod(2x)]=15015, \: \: 1 \leq x \leq 15015,$ x  is an integer.
\end{center}
The solutions will be $x = 1, x = 3, x = 5, x=7, x=11, x=13, x=15, x = 21.$ See Table $1$ for the list of solutions using Excel.
\end{example}
\begin{table}[h]
\caption {\textit{Example 5 solutions using  Excel:} The primes $x_i$ where $f(n - 2, x_i) -  n = 0$ are prime factors of $n$.} 
\centering 
\begin{tabular}{c ccc} 
\hline\hline 
$n$ & $\:\:\:\:\: x_i $ & $ f(n -2, x) - n$ & \\ [0.5ex] 
\hline 
$15015$  &$\textbf{ \:\:\:1}$ & $\textbf{0}$\\ 
$$ &$\textbf{ \:\:\:3}$ & $\textbf{0}$\\ 
$$ &$\textbf{\:\:\: 5}$ & $\textbf{0}$\\
$$ & $\textbf{\:\:\: 7}$ & $\textbf{0}$\\
$$ & \:\:\: 9 & 6\\
$$ & $\textbf{\:\:\:11}$ & $\textbf{0}$\\
$$ & $\textbf{\:\:\:13}$ & $\textbf{0}$\\
$$ & $\textbf{\:\:\:15}$ & $\textbf{0}$\\
$$ &\:\:\:17 & 30\\
$$ &\:\:\:19 & 14\\
$$ & $\textbf{\:\:\:21}$ & $\textbf{0}$\\
$$ &\:\:\:23 & 4\\
$$ &\:\:\:...&...\\
$$ & $\textbf{\:\:\:15015}$ & $\textbf{0}$ \\
\hline 
\end{tabular}
\label{table:nonlin} 
\end{table}
\begin{example}
Find the prime factors of $n = 7663 = 79*97$.\\
$\textbf{Solution:}$  Given $n = 7663$, 
the integer solutions of  $f(n -2, x) = n$  are all factors of $n$. 
 Enter the following equation:
\begin{center}
$7661 + 2x - [(7661 - x) mod(2x)] - 7663$ , $1 \leq x \leq 7663$, $x$ is an integer.
\end{center}
The solutions will be $x = 1$, $x = 79$, $x = 97$, $x = 7663$. 
See also Table $2$ on page $7$ for all solutions using Excel.
\end{example}
\begin{table}[h]
\caption {\textit{Example 6 solutions using  Excel:} The primes $x_i$ where $f(n - 2, x_i) -  n = 0$ are prime factors of $n$.}  
\centering 
\begin{tabular}{c ccc} 
\hline\hline 
$n$ & $\:\:\:\:\: x_i $ & $ f(n -2, x) - n$ & \\ [0.5ex] 
\hline 
$7663$  &$\textbf{\:\:\:1}$ & $\textbf{0}$\\ 
$$  & \:\:\:3 & 2\\
$$ & \:\:\:5 & 2\\
$$ & \:\:\:7 & 2\\
$$ & \:\:\:9 & 14\\
$$ &\:\:\:...&...\\
$$ & \:\:\:77 & 144\\
$$ &$\textbf{\:\:\:79}$ & $\textbf{0}$\\
$$ & \:\:\:81 & 32\\
$$ &\:\:\:...&...\\
$$ & \:\:\:95 & 32\\
$$ &$\textbf{\:\:\:97}$ & $\textbf{0}$\\
$$ & \:\:\:99 & 158\\
$$ &\:\:\:...&...\\
$$ &$\textbf{\:\:\:7663}$ & $\textbf{0}$\\
\hline 
\end{tabular}
\label{table:nonlin} 
\end{table}
\subsection{$f(n,x)$ is an unconditional general-purpose deterministic primality test.}
\begin{Theorem}
A positive odd integer $n$ greater than $1$ is a prime if and only if the equation $f(n-2, x) = n$ has only two solutions $1$ and $n$.
\end{Theorem}
\begin{proof}
Suppose that a positive odd integer $n$ is a prime, by virtue of the definition of a prime number, $n$ has only two factors $1$ and $n$. But all factors of $n$ are the $x$ solutions of the equation $f(n-2, x) = n$ as a result of the factorization Theorem in subsection $3.1$. Thus, the equation $f(n-2, x) = n$ has only two solutions $1$, and $n$. \\
Conversely, suppose that the equation $f(n-2,x) = n$ has only two solutions $1$ and $n$. By the factorization Theorem in subsection $3.1$ on page $6$, these solutions $1$ and $n$ are all the factors of $n$. Hence, $n$ has only two factors $1$ and $n$, this means that $n$ is a prime as a result of the definition of a prime number. We conclude that an odd integer $n$ greater than $1$ is a prime if and only if the equation $f(n-2, x) = n$ has only two solutions $1$ and $n$.
\end{proof}

\begin{example}
Using "Walfram, mathematica",  is  $n = 139$ a prime? \\
$\textbf{Solution:}$ Given $n = 139$, 
 the integer solutions of $f(n -2, x) = n$  are all the positive integer factors of $n$ by the factorization Theorem in subsection $3.1$. 
Enter the following equation 
\begin{center}
$137 + 2x - [(137 - x) mod(2x)] = 139$, $1 \leq x \leq 139$, $x$ is an integer.
\end{center}
The solutions will be $x = 1, x = 139$. This means that $139$ has no factor other than $1$ and itself. Hence, $139$ is a prime. 
\end{example}

\begin{example}
is $n = 3913 = 7*13*43$ a prime? \\
$\textbf{Solution:}$ 
the integer solutions of  $f(n -2, x) = n$ are all the positive integer factors of $n$ by the factorization Theorem in subsection $3.1$. 
Enter the following equation, 
\begin{center}
$3911 + 2x - [(3911 - x) mod(2x)] = 3913$ , $1 \leq x \leq 3913$, $x$ is an integer.
\end{center} 
The solutions are $x = 1, x= 7,  x = 13, x = 43, x = 91, x = 301, x = 559, x = 3913$.\\ Thus, $3913$ is not a prime. Another way to have all the solutions is to  make a table, choose $x$ such that $1 \leq x \leq n$ and collect the values of $x$ for which the equation $f(n - 2, x) = n$ holds as we did in Examples $5$ and $6$ using Excel.  \\
\end{example}
\subsection{Discussion on runtime analysis}
The function $f(n,x)$ has a runtime that scales poorly compared to some other primality test algorithms such as the AKS algorithm \cite{agrawal2004primes}, but it has some attractions and similarities with the trial division. \\
\textbullet\: The function $f(n,x)$ is an unconditional general-purpose deterministic primality test.\\
\textbullet\: The function $f(n,x)$ is a standalone algebraic function not a step by step finite sequence of logical instructions like an algorithm.\\
\textbullet \: The function $f(n,x)$ performs both, primality test and integer factorization.  (It is similar to the trial division in that sense.)\\
\textbullet \: The function $f(n,x)$ gives the exact value of $\omega(n)$ the number of distinct prime factors of $n$ (similar to the trial division). Indeed the $\omega(n)$ is equal to the number of prime solutions of the equation $f(n-2, x) = n$ as proven in subsection $3.1$ and illustrated by examples $5$ and $6$. \\
\textbullet \: The implementation of $f(n,x)$ in a computer is very simple. "If there was an algebraic function equivalent to the trial division, then $f(n,x)$ would be very similar to that function." \\\\
\textbullet \: A major advantage of $f(n,x)$ over the trial division is that $f(n,x)$ is a closed form expression and it can be used as a counting function. In the following process, we shall demonstrate how $f(n,x)$ is used to count the number of odd integers in a given interval. This counting process facilitates a rise of a relation between $\pi(n^2)$, $\sum_{3 \leq p \leq n} \frac{1}{p}$, and a weaker version of $\omega(n)$. \\
\section{Some applications of $f(n,x)$ as a counting function}%
\subsection{Counting the odd integers in the interval $(n^2, (n+2)^2]$} 
This leads to an exact formula of $\sum_{3 \leq x_i \leq n} \frac{1}{x_i}$. We count the number of odd composites in the interval $(n^2, (n+2)^2]$, and we add the number of primes that are in the same interval to obtain the total number of odd integers in the interval $(n^2, (n+2)^2]$. The notation $\pi(\mathcal{L})$ denotes the "Lengendre's pie": $\pi(\mathcal{L}) = \pi((n+2)^2) - \pi(n^2)$. \\\\
Let $x_n$ be the largest prime less than or equal to the ceiling of square root of $n$. The number of odd composites in the interval $(n^2, (n+2)^2]$ is the number of odd composite multiples of $3$ plus the number of odd composite multiples of $5$ ,..., plus the number of odd composite multiples of $x_n$, and we subtract the duplicates from the sum. Here is an example of duplicate counts: Say an odd integer $n_k = {p_1}^{{\alpha}_1}{p_2}^{{\alpha}_2}{p_3}^{{\alpha}_3}$ has three distinct prime factors $3 \leq p_1, p_2, p_3 \leq n$, so $n_k$ is counted three times. First, $n_k$ was counted as a multiple of $p_1$, then as a multiple of $p_2$ and finally as a multiple of $p_3$; thus, it is necessary to subtract the two extra times that the integer $n_k$ was counted. \\
How to count the odd composite multiples of $3$ that are in the interval $(n^2, (n+2)^2]$ ? These numbers start from the smallest odd composite multiple of $3$ greater than ${n}^2$  to the largest odd multiple of $3$ less than or equal to ${(n+2)}^2$. The function $f(n,x)$ can compute the smallest odd composite multiple of $3$ greater than ${n}^2$ that is $f({n}^2, 3)$. It is also known that (the largest odd composite multiple of $3$ less than or equal to ${(n+2)}^2$) is equal to (the smallest odd composite multiple of $3$ that is greater than ${(n+2)}^2) - 6$ this is $f({(n+2)}^2, 3) - 6$. So between the first odd composite multiple of $3$ greater than $n^2$, and the last odd composite multiple of $3$ less than or equal to $(n+2)^2$, every time we move $6$ units, there is an odd multiple of $3$.
In conclusion, the number of odd composite multiples of $3$ in the interval $(n^2, (n+2)^2]$ is, 
\[\frac{[f({(n+2)}^2, 3) - 6] - [f({n}^2, 3)]}{6} + 1 = \frac{f({(n+2)}^2, 3) - f({n}^2, 3)}{6}.\]
Similarly, the number of odd composite multiples of $5$ in the interval $(n^2, (n+2)^2]$ is,
\[\frac{[f({(n+2)}^2, 5) - 10] - [f({n}^2, 5)]}{10} + 1 =  \frac{f({(n+2)}^2, 5) - f({n}^2, 5)}{10}.\]
Hence, in general, given a prime $x_i$ such that $3 \leq x_i \leq n$, the number of odd composite multiples of $x_i$ that are in the interval $(n^2, (n+2)^2]$ is, \\
 \[ \frac{f({(n+2)}^2, x_i) - f({n}^2, x_i)}{2x_i}.\]
Thus, the number of odd composites in the interval $(n^2, (n+2)^2]$ is,\\
 \[\left(\sum_{3 \leq x_i \leq n} \frac{f({(n+2)}^2, x_i) - f({n}^2, x_i)}{2x_i} \right) - dup.\]
 \[dup \: = \sum_{{n}^2 < n_k \leq {(n+2)}^2} dup(n_k).\]
\[
    {dup(n_k)}_{{n}^2 < n_k \leq {(n+2)}^2} = 
\begin{cases}
    0,& \text{if  m = 0}\\
    m-1,              & \text{otherwise}
\end{cases}
\: \: \: \: \: \: \:\: \: \:\: \: \]
where $m$ is the number of distinct prime factors less than or equal to $n$, of the odd integer $n_k$. Note: $m$ is simply a weaker version of $\omega(n_k)$.\\\\
Let's give a more comprehensive explanation of the $dup$ function. Suppose that $n_k$ is an odd integer in the interval $({n}^2, {(n+2)}^2]$ and $n_k = {p_1}^{s_1}{p_2}^{s_2}...{\: p_r}^{s_r}$ and $n_k$ has $m$ distinct prime factors less than or equal to $n$, say these prime factors are $p_1, p_2 ,..., p_m$,  then $n_k$ was counted as a multiple $p_1$, a multiple of $p_2$ ,..., and as a multiple of $p_m$. Thus, $n_k$ was duplicated $m - 1$ times. Since $f(n,x)$ does not filter duplicate items, it is necessary to subtract the extra number of times $n_k$ is counted. The sum of these extra number of times each $n_k$ was counted, is the $dup$ function.\\\\
The total number of odd integers in the interval $(n^2, (n+2)^2]$ is,
\[\left(\sum_{3 \leq x_i \leq n} \frac{f({(n+2)}^2, x_i) - f({n}^2, x_i)}{2x_i} \right) - dup + \pi(\mathcal{L}).\]
Where $\pi(\mathcal{L})$ is the "Legendre's Pie". i.e. $\pi(\mathcal{L}) = \pi({(n+2)}^2) - \pi({n}^2)$. 
The total number of odd integers in the interval $(n^2, (n+2)^2]$  is also, 
\[\frac{{(n+2)}^2 - {n}^2}{2}.\]
\textbf{Note:} The function $f(n,x)$ does not always include $(n+2)^2$ in the count of the number of composites in the interval $(n^2, (n+2)^2]$. We shall explain in detail in the following two cases why the inclusion of $(n+2)^2$ in the count depends on the primality of $n+2$. \\\\
\textbullet \: Case 1: If $n+2$ is a composite,\\
then $f(n,x)$ \textbf{will count} $(n+2)^2$ as a composite in the interval $(n^2, (n+2)^2]$. This is because $(n+2)$ is a composite, so $n+2$ is a multiple of some prime $p \leq n$, and $(n+2)^2$ is also a multiple of $p$. Now since $x_i$ is a prime such that $3 \leq x_i \leq n$, this means that $x_i$ will take the value of $p$ at some point. Moreover, because $f(n, x)$ counts all the multiple of $x_i$ that are in the interval $(n^2, (n+2)^2]$; it follows that $(n+2)^2$ will be counted as a multiple of $x_i = p$. Thus, if $(n+2)$ is a composite, then we can say that the number of odd integers in the interval $(n^2, (n+2)^2]$ that $f(n,x)$ has counted is equal to the expected number of odd integers in the interval $(n^2, (n+2)^2]$. This expected number is $\frac{{(n+2)}^2 - {n}^2}{2}.$ The comparison translates to,
\begin{equation}
\left(\sum_{3 \leq x_i \leq n} \frac{f({(n+2)}^2, x_i) - f({n}^2, x_i)}{2x_i} \right) - dup + \pi(\mathcal{L}) = \frac{{(n+2)}^2 - {n}^2}{2}.
\end{equation}
\textbullet \: Case 2: If $n+2$ is a prime,\\
then $f(n,x)$ \textbf{will not count} $(n+2)^2$ as a composite in the interval $(n^2, (n+2)^2]$. This is because the prime $x_i$ is such that $3 \leq x_i \leq n$. This means that $x_i$ cannot be equal to $n+2$ which is the only prime factor of $(n+2)^2$. Hence, $(n+2)^2$ will not be counted as a multiple of any prime $x_i \leq n$. Therefore, we must \textbf{Add 1} to the number of odd integers counted by $f(n,x)$ to compensate the missing count of $(n+2)^2$. So if $n+2$  is a prime, then the number of odd integers in the interval $(n^2, (n+2)^2]$ that $f(n,x)$ has counted \textbf{+1}, is equal to the expected number of odd integers in the interval $(n^2, (n+2)^2]$, which is equal to, $\frac{{(n+2)}^2 - {n}^2}{2}.$  This gives rise to,
\begin{equation}
\left(\sum_{3 \leq x_i \leq n} \frac{f({(n+2)}^2, x_i) - f({n}^2, x_i)}{2x_i} \right) - dup + \pi(\mathcal{L}) + 1= \frac{{(n+2)}^2 - {n}^2}{2}.
\end{equation}
Equations $(4.1)$ and $(4.2)$ differ only by $1$. These two equations will merge with a help of an $\epsilon$ variable that takes the value of $0$ or $1$ depending on the primality of $n+2$. But first, let's simplify equation $(4.2)$ where $f({(n+2)}^2, x_i) = {(n+2)}^2 + 2x_i - ({(n+2)}^2 - x_i)mod(2x_i)$, and $f({n}^2, x_i) =  {n}^2 + 2x_i - ({n}^2 - x_i) mod(2x_i)$. The result after inserting these values into equation $(4.2)$ is, 
\begin{equation}
\sum_{3 \leq x_i \leq n} \frac{{(n+2)}^2 - {n}^2}{2x_i} +  \sum_{3 \leq x_i \leq n} \frac{({n}^2 - x_i) mod(2x_i) -  ({(n+2)}^2 - x_i) mod(2x_i)}{2x_i}- 
\end{equation}
$ - \: dup + \pi(\mathcal{L}) = \frac{{(n+2)}^2 - {n}^2}{2} - 1.$ 
\[\textup{ For simplification, let } C_i =  \frac{({(n+2)}^2 - x_i) mod(2x_i) -  ({n}^2 - x_i) mod(2x_i)}{2x_i}.\]
\[\textup{Hence,} \sum_{3 \leq x_i \leq n} \frac{({n}^2 - x_i) mod(2x_i) -  ({(n+2)}^2 - x_i) mod(2x_i)}{2x_i} = - \sum_{3 \leq x_i \leq n} C_i.\]
Incorporate $C_i$ into equation $(4.3)$, and the result is, 
\[\frac{{(n+2)}^2 - {n}^2}{2} \left(\sum_{3 \leq x_i \leq n} \frac{1}{x_i}  - 1 \right) -  \sum_{3 \leq x_i \leq n} C_i  = dup - \pi(\mathcal{L}) - 1.\]
This means that, \\
\begin {equation}
{\sum_{3 \leq x_i \leq n} \frac{1}{x_i} = 2\left(\frac {dup + \left (\sum_{3 \leq x_i \leq n} C_i \right)  - \pi(\mathcal{L}) - 1}{{(n+2)}^2 - {n}^2} \right) + 1.} 
\end{equation}
equation $(4.4)$ comes from equation $(4.3)$ which is a specific case when $n+2$ is a prime. Case $1$ and $2$ are merged into a general formula,\\
\begin{center}
\fbox{ 
\begin{minipage}{4.5 in}
\begin {equation}
{\sum_{3 \leq x_i \leq n} \frac{1}{x_i} = 2\left(\frac {dup + \left (\sum_{3 \leq x_i \leq n} C_i \right)  - \pi(\mathcal{L}) - \epsilon}{{(n+2)}^2 - {n}^2} \right) + 1.}
\end{equation}
\[
    \epsilon = 
\begin{cases}
    1,& \text{if  n+2 is prime}\\
    0,              & \text{otherwise}
\end{cases}
\]
\end{minipage}}
\end{center} 
\vspace{3mm}
Equation $(4.5)$ gives an exact formula that establishes a relation between the partial sum of the reciprocals of odd prime numbers $\sum_{3 \leq x_i \leq n} \frac{1}{x_i}$ and $\pi(\mathcal{L})$. This improved upon previous works of other authors, who established the upper and lower bounds of the partial sum of the reciprocals of prime numbers. Moreover, equation $(4.5)$ opens a new window into the territories of the Legendre's Conjecture.\\
\begin{example}
What is the partial sum of the reciprocals of odd primes less than or equal to 11?\\
\textbf{Solution:} Let $n = 11$, so $(n+2) = 13$ is a prime. 
Using equation $(4.5)$ where $n+2$ is prime means that $\epsilon = 1$. To have a better view, here is the list of odd composites between ${11}^2$ and ${13}^2$.\\
\begingroup
\obeylines
\textbf{11*11=121},  (123 = 3, 41),  (125 = 5, 5, 5),  (129 = 3, 43),  (133 = 7, 19),  (135 = 3, 3, 3, 5),  (141 = 3, 47),  (143 = 11, 13),  (145 = 5, 29),  (147 = 3, 7, 7),  (153 = 3, 3, 17),  (155 = 5, 31),  (159 = 3, 53),  (161 = 7, 23),  (165 = 3, 5, 11),  \textbf{13*13=169}. \\%
\endgroup%
$dup = \sum_{{11}^2 <  n_k \leq {13}^2} dup(n_k) = dup(123) + dup(125) + ... + dup(169).$
\[
    {dup(n_k)}_{{11}^2 < n_k \leq {13}^2} = 
\begin{cases}
    0,& \text{if  m = \textup{0}}\\
    m-1,              & \text{otherwise}
\end{cases}
\: \: \: \: \: \: \:\: \: \:\: \: \]
where $m$ is the number of distinct prime factors not exceeding $n=11$, of the odd integer $n_k$.\\
\textbf{Note:} From the definition of $dup(n_k)$, we can see that, if $n_k$ is a prime or a power of a prime, then $dup(n_k) = 0$. So we can skip $n_k$, if $n_k = p^{\alpha}$ where $\alpha\geq1$ is an integer. \\\\
$dup(123 = 3*41) = 1 - 1 = 0$ because $123$ has only one distinct prime factor less than or equal to $11$. Thus, $123$ has no duplicate count. It was only counted once as a multiple of 3.\\
$dup(125=5*5*5) = 1- 1 = 0$ because $125$ has only one distinct prime factor that is less than or equal to $11$, so $125$ also was not counted more than once. \\
...\\
$dup(135=3*3*3*5) = 2 -1 = 1$ because $135$ has two distinct prime factors that are less than or equal to $11$, these factors are $3$ and $5$. Hence, $135$ was counted by $f(n, x)$ as a multiple of $3$ and also as a multiple of $5$. For these reasons, the extra count on $135$ must be eliminated. Similarly, \\
$dup(147= 3*7*7) = 2-1 = 1.$\\
...\\
$dup(165=3*5*11) = 3 - 1 = 2$. \\
$dup(169=13*13) = 0$. the integer $169$ has no prime factor less than or equal to $n= 11$.\\\\
\textbullet {\: $ dup = \sum_{{11}^2 <  n_k \leq {13}^2} dup(n_k)  =  1 + 1 + 2 = \textbf{4}$.} \\
\textbullet{ \: $\pi(\mathcal{L}) = \pi({(n+2)}^2) - \pi({n}^2) = \textbf{9}.$ }\\
\textbullet{\:  $ \sum_{3 \leq x_i \leq n} C_i = \sum_{3 \leq x_i \leq n} \frac{({(n+2)}^2 - x_i) mod(2x_i) -  ({n}^2 - x_i) mod(2x_i)}{2x_i}$} $ = \frac{4 - 4}{2*3} + \frac{4 - 6}{2*5} + \frac{8 - 2}{2*7} + \frac{4 - 0}{2*11} = 0 - \frac{1}{5} + \frac{3}{7} + \frac{2}{11} = \textbf{ 0.41039}.$ \\
\textbullet{\: ${(n+2)}^2 - {n}^2 = {13}^2 - {11}^2 = 169 - 121 = \textbf{48}.$} \\
\textbullet{\: $\epsilon = \textbf{1}.$}\\
$2\left(\frac {dup + \left(\sum_{3 \leq x_i \leq n} C_i \right) - \pi(\mathcal{L})- \epsilon}{{(n+2)}^2 - {n}^2} \right) + 1 = 2\left(\frac {4 + (0.41039) - 9 - 1}{48} \right) + 1.$ \\
$2\left(\frac {dup + \left(\sum_{3 \leq x_i \leq n} C_i \right) - \pi(\mathcal{L})- \epsilon}{{(n+2)}^2 - {n}^2} \right) + 1 = \textbf{0.7671}.$ 
Therefore, by equation $(6)$, $\sum_{3 \leq x_i \leq n} \frac{1}{x_i} = \fbox{\textbf{0.7671}.}$ Let's verify our result by calculating the value on the left side of equation $(4.5)$.\\
$\sum_{3 \leq x_i \leq n} \frac{1}{x_i} = \sum_{3 \leq x_i \leq 11} \frac{1}{x_i} = \frac{1}{3} + \frac{1}{5}+\frac{1}{7}+\frac{1}{11} = \frac{886}{1155} = \fbox{\textbf{0.7671}.}$
\end{example}
\begin{example}
What is the partial sum of the reciprocals of odd primes less than or equal to 40?\\
\textbf{Solution:} Set $n = 37$ the largest prime less than or equal to $40$, then $n + 2 = 39$. 
Using equation $(4.5)$ where $n+2$ is a composite gives $\epsilon = 0$.\\ Here is the list of the odd composite numbers between ${37}^2$ and ${39}^2$.\\
\begingroup
\obeylines
\textbf{37*37=1369},  (1371 = 3, 457),  (1375 = 5, 5, 5, 11),  (1377 = 3, 3, 3, 3, 17),  (1379 = 7, 197),  (1383 = 3, 461),  (1385 = 5, 277),  (1387 = 19, 73),  (1389 = 3, 463),  (1391 = 13, 107),  (1393 = 7, 199),  (1395 = 3, 3, 5, 31),  (1397 = 11, 127),  (1401 = 3, 467),  (1403 = 23, 61),  (1405 = 5, 281),  (1407 = 3, 7, 67),  (1411 = 17, 83),  (1413 = 3, 3, 157),  (1415 = 5, 283),  (1417 = 13, 109),  (1419 = 3, 11, 43),  (1421 = 7, 7, 29),  (1425 = 3, 5, 5, 19),  (1431 = 3, 3, 3, 53),  (1435 = 5, 7, 41),  (1437 = 3, 479),  (1441 = 11, 131),  (1443 = 3, 13, 37),  (1445 = 5, 17, 17),  (1449 = 3, 3, 7, 23),  (1455 = 3, 5, 97),  (1457 = 31, 47),  (1461 = 3, 487),  (1463 = 7, 11, 19),  (1465 = 5, 293),  (1467 = 3, 3, 163),  (1469 = 13, 113),  (1473 = 3, 491),  (1475 = 5, 5, 59),  (1477 = 7, 211),  (1479 = 3, 17, 29),  (1485 = 3, 3, 3, 5, 11),  (1491 = 3, 7, 71),  (1495 = 5, 13, 23),  (1497 = 3, 499),  (1501 = 19, 79),  (1503 = 3, 3, 167),  (1505 = 5, 7, 43),  (1507 = 11, 137),  (1509 = 3, 503),  (1513 = 17, 89),  (1515 = 3, 5, 101),  (1517 = 37, 41),  (1519 = 7, 7, 31),  \textbf{39*39 =1521}.
\endgroup%
\[dup = \sum_{{37}^2 <  n_k \leq {39}^2} dup(n_k) = dup(1371) + dup(1375) + ... + dup(1521)\]
\[
    {dup(n_k)}_{{37}^2 < n_k \leq {39}^2} = 
\begin{cases}
    0,& \text{if  m = 0}\\
    m-1,              & \text{otherwise}
\end{cases}
\: \: \: \: \: \: \:\: \: \:\: \: \]
where $m$ is the number of distinct prime factors not exceeding $n=37$, of the odd integer $n_k$.\\
In this example, only the $dup(n_k)$ that are not equal to zero are computed to avoid lengthy calculations. \\\\
$dup(1375=5*5*5*11) = 2 - 1 = 1$ because $1375$ has two distinct prime factors that are less than or equal to $37$. \\
$dup(1377=3*3*3*17) = 2 - 1 = 1$ because $1377$ has two distinct prime factors that are less than or equal to $37$. \\
$dup(1395=3*3*5*31) = 3 - 1 = 2$ because $1395$ has three distinct prime factors that are less than or equal to $37$. \\
$dup(1407=3*7*67) = 2 - 1 = 1$ \\ 
$dup(1419=3*11*43) = 2 - 1 = 1$ \\ 
$dup(1421=7*7*29) = 2 - 1 = 1$\\ 
$dup(1425=3*5*5*19) = 3 - 1 = 2$\\ 
$dup(1435=5*7*41) = 2 - 1 = 1$\\ 
$dup(1443=3*13*37) = 3 - 1 = 2$\\ 
$dup(1445=5*17*17) = 2 - 1 = 1$\\ 
$dup(1449=3*3*7*23) = 3 - 1 = 2$\\ 
$dup(1455=3*5*97) = 2 - 1 = 1$\\ 
$dup(1463=7*11*19) = 3 - 1 = 2$ \\ 
$dup(1479=3*17*29) = 3 - 1 = 2$ \\ 
$dup(1485=3*3*3*5*11) = 3 - 1 = 2$\\ 
$dup(1491=3*7*71) = 2 - 1 = 1$ \\ 
$dup(1495=5*13*23) = 3 - 1 = 2$ \\ 
$dup(1505=5*7*43) = 2 - 1 = 1$ \\ 
$dup(1515=3*5*101) = 2 - 1 = 1$ \\ 
$dup(1519=7*7*31) = 2 - 1 = 1$ \\ 
$dup(1521=3*3*13*13) = 2 - 1 = 1$ \\\\ 
\textbullet {\: $ dup = \sum_{{37}^2 <  n_i \leq {39}^2} dup(n_i) = \textbf{29}$.} \\
\textbullet{ \: $\pi(\mathcal{L}) = \pi({(n+2)}^2) - \pi({n}^2) = \textbf{21}.$ }\\
\textbullet{\:  $ \sum_{3 \leq x_i \leq n} C_i = \sum_{3 \leq x_i \leq n} \frac{({(n+2)}^2 - x_i) mod(2x_i) -  ({n}^2 - x_i) mod(2x_i)}{2x_i}$} $ = \frac{0 - 4}{2*3} + \frac{6 - 4}{2*5} + \frac{2 - 4}{2*7} + \frac{14 - 16}{2*11} + \frac{0-4}{2*13} + \frac{8-26}{2*17}+ \frac{20-20}{2*19}+ \frac{26-12}{2*23}+ \frac{42-6}{2*29}+ \frac{2-36}{2*31}+ \frac{4-0}{2*37} = \textbf{ -0.952986}.$ \\
\textbullet{\: ${(n+2)}^2 - {n}^2 = {13}^2 - {11}^2 = 169 - 121 = \textbf{152}$.}\\
\textbullet{\: $\epsilon = \textbf{0}$.} \\\\
$2\left(\frac {dup + \left(\sum_{3 \leq x_i \leq n} C_i \right) - \pi(\mathcal{L}) - \epsilon}{{(n+2)}^2 - {n}^2} \right) + 1 = 2\left(\frac {29 + ( -0.952986) - 21 - 0}{152} \right) + 1.\\
2\left(\frac {dup + \left(\sum_{3 \leq x_i \leq n} C_i \right) - \pi(\mathcal{L}) - \epsilon}{{(n+2)}^2 - {n}^2} \right) + 1  = \textbf{1.09272}$.\\
Therefore, by equation $(4.5)$, $\sum_{3 \leq x_i \leq n} \frac{1}{x_i} = \fbox{\textbf{1.09272}.}$ Let's verify our result by calculating the value on the left side of equation $(4.5)$.
\[\sum_{3 \leq x_i \leq n} \frac{1}{x_i} = \sum_{3 \leq x_i \leq 37} \frac{1}{x_i} = \frac{1}{3} + \frac{1}{5}+...+\frac{1}{37} = \frac{4054408822031}{3710369067405} = \fbox{\textbf{1.09272}.}\]
\end{example}
\vspace{3mm}
\[\textbf{More Examples without detailed calculations.}\]
\begingroup
\obeylines


\textbf{Example:} Compute the partial sum of reciprocals of odd primes not exceeding n = 41. \\
\textbullet{\:  $ \sum_{3 \leq x_i \leq n} C_i  = -0.162415.$}
dup = 31.
$\pi(\mathcal{L})= 20.$
$(43^2 - 41^2) = 168.$
$\epsilon = {1}.$

 Sum from  the right side of equation $4.5$ is,  \textbf{1.11711}.
  $\sum_{3 \leq x_i \leq 41} \frac{1}{x_i} = \frac{1}{3} + \frac{1}{5}+...+\frac{1}{41} = $\textbf{1.11711}. \\

\textbf{Example:} Compute the partial sum of reciprocals of odd primes not exceeding n = 43. \\
\textbullet{\:  $ \sum_{3 \leq x_i \leq n} C_i  = -1.64745.$}
dup = 37.
$\pi(\mathcal{L}) = 23.$
$(45^2 - 43^2) = 176.$
$\epsilon = {0}.$

 Sum from the right side of equation $4.5$ is, \textbf{1.14037}.
 $\sum_{3 \leq x_i \leq 43} \frac{1}{x_i} = \frac{1}{3} + \frac{1}{5}+...+\frac{1}{43} = $ \textbf{1.14037.} \\

\textbf{Example:} Compute the partial sum of reciprocals of odd primes not exceeding n = 45. \\
\textbullet{\:  $ \sum_{3 \leq x_i \leq n} C_i  = 1.91403.$}
dup = 35.
$\pi(\mathcal{L}) = 23.$
$(47^2 - 45^2) = 184.$
$\epsilon = 1.$
Sum from  the right side of equation $4.5$ is, \textbf{1.14037.}
$\sum_{3 \leq x_i \leq 45} \frac{1}{x_i} = \frac{1}{3} + \frac{1}{5}+...+\frac{1}{43}  = $\textbf{1.14037.}%
\endgroup%
\subsection{Counting the odd integers in the interval $(n^2, n(n+2)]$} 
This count give an exact formula of $\pi(\mathcal{L})_1 = \pi(n(n+2))-\pi(n^2)=\pi(n+1)^2-\pi(n^2)$, the number of primes between two consecutive squares when starting from an odd square $n^2$.\\\\
\textbf{Note:} The number of primes in the interval  $(n^2, n(n+2)]$ is the same as the number of prime in the interval  $(n^2, (n+1)^2]$ since  $ n(n+2)+1 = (n+1)^2$ and $(n+1)^2$ is an even number. We use $n(n+2)$ because $f(n,x)$ takes only odd integers. \\\\
The new equation is built by counting the odd integers in the interval $(n^2, n(n+2)]$. To avoid repeating the process of equation $(4.5)$, we simply deduce this new equation from $(4.5)$. We replace the upper bound $(n+2)^2$ in equation $(4.5)$ with the new upper bound $n(n+2)$. The variables in the new equation have similar definitions to those in equation $(4.5)$, but the interval here is $(n^2$, $n(n+2)]$ instead of $(n^2$, $(n+2)^2].$  
\[
    {dup(n_k)}_{{n}^2 < n_k \leq {n(n+2)}} = 
\begin{cases}
    0,& \text{if  m = 0}\\
    m-1,              & \text{otherwise}
\end{cases}
\: \: \: \: \: \:\]
where $m$ is the number of distinct prime factors not exceeding $n$ of the odd integer $n_k$.
\[(dup)_1 \: := \sum_{{n}^2 < n_k \leq {n(n+2)}} dup(n_k).\]
\[\pi(\mathcal{L})_1 = \pi(n(n+2)) - \pi(n^2).\]
\[C_1 =  \frac{({n(n+2)} - x_i) mod(2x_i) - (n^2 - x_i) mod(2x_i) }{2x_i}.\]
\begin{center}
\fbox{ 
\begin{minipage}{4.5 in}
\begin {equation}
{\sum_{3 \leq x_i \leq n} \frac{1}{x_i} = 2\left(\frac {(dup)_1 + \left (\sum_{3 \leq x_i \leq n} C_1 \right)  - \pi(\mathcal{L})_1}{{n(n+2)} - {n}^2} \right) + 1.}
\end{equation}
\end{minipage}}
\end{center} 
\textup{\\}
\subsection{Counting the odd integers in the interval $(n(n+2), (n+2)^2]$}
This count gives an exact formula of $\pi(\mathcal{L})_2 = \pi((n+2)^2)-\pi(n(n+2))=\pi((n+2)^2)-\pi((n+1)^2)$, the number of primes between two consecutive squares when starting from an even square $(n+1)^2$.\\
The equation obtained by counting the odd integers in the interval $(n(n+2), (n+2)^2]$ is deduced from $(4.5)$ by replacing $n^2$ with $n(n+2)$.The variables in the new equation have similar definitions to those in $(4.5)$. 
\[
    {dup(n_k)}_{n(n+2) < n_k \leq {(n+2)}^2} = 
\begin{cases}
    0,& \text{if  m = 0}\\
    m-1,              & \text{otherwise}
\end{cases}
\: \: \: \: \: \: \:\: \: : \: \: \]
where $m$ is the number of distinct prime factors not exceeding $n$ of the odd integer $n_k$.
\[(dup)_2 \: := \sum_{{n(n+2)} < n_k \leq {(n+2)^2}} dup(n_k).\]
\[\pi(\mathcal{L})_2 = \pi((n+2)^2) - \pi(n(n+2)).\]
\[C_2 =  \frac{({(n+2)^2} - x_i) mod(2x_i) - (n(n+2) - x_i) mod(2x_i) }{2x_i}.\]
\begin{center}
\fbox{ 
\begin{minipage}{4.5 in}
\begin {equation}
{\sum_{3 \leq x_i \leq n} \frac{1}{x_i} = 2\left(\frac {(dup)_2 + \left (\sum_{3 \leq x_i \leq n} C_2 \right)  - \pi(\mathcal{L})_2 - \epsilon}{{(n+2)}^2 - {n(n+2)}} \right) + 1.}
\end{equation}

\[
    \epsilon = 
\begin{cases}
    1,& \text{if  n+2 is prime}\\
    0,              & \text{otherwise}
\end{cases}
\]
\end{minipage}}
\end{center} 
\textup{\\}
\subsection{A possible path to the Legendre's Conjecture} 
We obtain an exact formula of the partial sum of the receiprocals of odd primes in equation $(4.6)$. This is in terms of $\pi(\mathcal{L})_1$ the number of primes between $n^2$ and $(n+1)^2$.  One can argue that, \\
\textbullet \: If $\pi(\mathcal{L})_1 = 0$, then  $\sum_{3 \leq x_i \leq n} \frac{1}{x_i}$ would be bigger than its own upper bound. And that is impossible. Hence, there must be a prime between $n^2$ and $(n+1)^2$.  But one would need a good upper bound and a descent approximation of $(dup)_1$. A similar reasonment can be made with $\pi(\mathcal{L})_2$ in equation $(4.7)$. 

\section{Refinement of the prime counting function}
This is another application of $f(n,x)$ as a counting function. 
A refinement of the prime counting function leads to a better understanding of the distribution of prime numbers. It also increases the scope to access and wrestle down some open problems in prime numbers. \\
The correctness and precision of the results in examples $10$ and $11$, echo on the strength of equation (4.5). We shall build a similar equation by counting the odd integers in the interval $(1, n^2]$ instead of $(n^2, (n+2)^2]$. \\\\
Let $x_n$ denote the largest primes not exceeding the ceiling of the square root of $n$. The number of odd composites in the interval $(1, n^2]$ is the number of odd composite multiples of $3$, plus the number of odd composite multiples of $5$ ..., plus the number of odd composite multiples of $x_n$, and we subtract the duplicates.
The odd composite multiples of $3$ not exceeding ${n}^2$ start from $3$ where $3$ is not included, and extend to the largest odd multiple of $3$ not exceeding ${n}^2$ which is given by $n^2 - (n^2 - 3) mod (2*3)$. Thus, the number of odd composite multiples to $3$ not exceeding $n^2$ is,\\
\[\frac{n^2 - (n^2 - 3) mod (2*3) -  3}{2*3}.\]
Similarly, the number of odd composite multiples of $5$ less than or equal to ${n}^2$ is,  \\
\[\frac{n^2 - (n^2 - 5) mod (2*5) -  5}{2*5}.\]
In general, given a prime $x_i$ such that $3 \leq x_i \leq n$ the number of odd composite multiples of $x_i$ not exceeding ${n}^2$ is,\\
\[\frac{n^2 - (n^2 - x_i) mod (2x_i) -  x_i}{2x_i}.\]
Consequently, the number of odd composites greater than $1$ and not exceeding ${n}^2$ is,\\
\[\left(\sum_{3 \leq x_i \leq n} \frac{n^2 - (n^2 - x_i) mod (2x_i) -  x_i}{2x_i} \right) - dup.\]
 \[dup \: = \sum_{1 < n_k \leq n^2} dup(n_k).\]
\[ \: \:\: \: \:\: \: 
    {dup(n_k)}_{1 < n_k \leq n^2} = 
\begin{cases}
    0,& \text{if  m = 0}\\
    m-1,              & \text{otherwise}
\end{cases}
\: \: \: \: \: \: \:\: \: \:\: \: \]
where $m$ is the number of distinct prime factors less than or equal to $n$, of the odd integer $n_k$. \\\\
The total number of odd integers in the interval $(1, n^2]$ is,\\
\[\left(\sum_{3 \leq x_i \leq n} \frac{n^2 - (n^2 - x_i) mod (2x_i) -  x_i}{2x_i} \right) - dup + \pi({n^2}).\]
\[\sum_{3 \leq x_i \leq n} \frac{n^2}{2x_i} - \left(\sum_{3 \leq x_i \leq n} \frac{(n^2 - x_i) mod (2x_i) +  x_i}{2x_i} \right) - dup + \pi({n^2}).\]
\[\textup{Set  \: } B_i =  \frac{(n^2 - x_i) mod(2x_i) + x_i}{2x_i}.\]
\[\sum_{3 \leq x_i \leq n} \frac{n^2}{2x_i} - \sum_{3 \leq x_i \leq n} B_i - dup + \pi({n^2}).\]
 The total number of odd integers in the interval $(1, n^2]$ is also,
\[\frac{n^2 - 1}{2}.\]
We obtain the equality,
\[\sum_{3 \leq x_i \leq n} \frac{n^2}{2x_i} - \sum_{3 \leq x_i \leq n} B_i - dup + \pi({n^2}) = \frac{n^2 -1}{2}.\]
This simplifies to, 
\[\sum_{3 \leq x_i \leq n} \frac{1}{x_i} = 2\left(\frac{dup + \left(\sum_{3 \leq x_i \leq n} B_i \right) - \pi({n^2}) - \frac{1}{2} }{n^2}\right) + 1.\]
\begin {equation}
\fbox{$\pi(n^2) = dup + \left(\sum_{3 \leq x_i \leq n} B_i \right) - \frac{1}{2} - \frac{n^2}{2} \left(\sum_{3 \leq x_i \leq n} \frac{1}{x_i}  - 1\right) .$}
\end{equation}
\textbf{\\Note:} The prime $2$ is not included in the count as we only worked with odd integers.\\ Equation $(5.1)$  is a refinement of the prime counting function it is an "exact formula", and it gives more detailed information. We can see that $\pi(n^2)$ is expressed in terms of some important values with historic interests such as the partial sum of the reciprocals of odd primes inferior or equal to $n$, and the $dup$ function which is derived from a well known function $\omega(n)$. 
\begin{example}
Compute $\pi(7^2)$.\\
\[dup \: = \sum_{1 < n_k \leq 49} dup(n_k) = dup(3) + dup(5)+dup(7)+dup(9)+...+dup(49).\]
By the definition of the $dup$ function, we know that if $n_k$ is a prime or a power of a prime, then $dup(n_k) = 0$. 
So we can skip $n_k$, if $n_k = p^{\alpha}$ where $\alpha\geq1$ is an integer. \\ 
$dup(15)=dup(3*5)=2-1=1.$\\
$dup(21)=dup(3*7)=2-1=1.$\\
$dup(33)=dup(3*11)=1-1=0.$\\
$dup(35)=dup(5*7)=2-1=1.$\\
$dup(39)=dup(3*13)=1-1=0.$\\
$dup(45)=dup(3*3*5)=2-1=1.$\\
\textbullet \: $dup \: = \sum_{1 < n_k \leq 49} dup(n_k) = 1+1+1+1= 4.$\\
\textbullet \: $\sum_{3 \leq x_i \leq n} B_i = \frac{(7^2 - 3) mod(2*3) + 3}{2*3} + \frac{(7^2 - 5) mod(2*5) + 5}{2*5}+ \frac{(7^2 - 7) mod(2*7) + 7}{2*7} = \frac{7}{6} +  \frac{9}{10}+  \frac{1}{2}.$\\
\textbullet \: $\sum_{3 \leq x_i \leq 7} \frac{1}{x_i} =  \frac{1}{3} +  \frac{1}{5} +  \frac{1}{7} .$
\begin{eqnarray*}
\pi(n^2) &=& dup + \left(\sum_{3 \leq x_i \leq n} B_i \right) - \frac{1}{2} - \frac{n^2}{2} \left(\sum_{3 \leq x_i \leq n} \frac{1}{x_i}  - 1\right) .\\
&=&  4+(\frac{7}{6} +  \frac{9}{10}+  \frac{1}{2}) - \frac{1}{2}-\frac{7^2}{2}(\frac{1}{3} +  \frac{1}{5} +  \frac{1}{7} -1).\\
\pi(7^2)&=&14. 
\end{eqnarray*}
\end{example}
\subsection{Comparing $\pi(x)$ with its refinement}
\textup{\\}
The prime number Theorem, first proved by Hadamard \cite{hadamard1896distribution} and de la Vallee-Poussin \cite{poussin1897distribution} states that \\
\[\pi(x) \sim \frac{x}{\ln x} \textup{ \: \: as x $\to \infty$}.\]
And in its strongest known form, the prime number Theorem is a statement that\\
\begin {equation}
\pi(x) = \textup{Li}(x) + R(x).
\end{equation}
where
\[\textup{Li}(x) := \int^x_2 \frac{dt}{\log t} = x \left(\frac{1}{\log x} + \frac{1!}{{\log}^2 x} + ... + \frac{m!}{{\log}^{m+1} x} + \mathcal{O}\left(\frac{1}{{\log}^{m+2} x}\right) \right) \] for any fixed integer $m \geq 0$ and 
\[R(x) \ll {x } exp\left( - C\delta(x)\right), \textup{ \: \:} \delta(x) := (\log x)^{3/5}(\log\log x)^{-1/5}     (C > 0) \textup{ \cite{hardy1979introduction}}.\]
\\ The function $f(n,x_i)$ where $n$ is a positive odd integer greater than $3$ and $x_i$ is an odd prime less than or equal to the ceiling of the square root of  $n$, states that,
\begin {equation}
\pi(n^2) = dup + \left(\sum_{3 \leq x_i \leq n} B_i \right) - \frac{1}{2} - \frac{n^2}{2} \left(\sum_{3 \leq x_i \leq n} \frac{1}{x_i}  - 1\right).
\end{equation}
 where 
 \[\textup{\: \: \: \: \: \: \: \: \: \: \:\: \: \: \: \: \: \: \:} dup \: = \sum_{1 < n_k \leq {n}^2} dup(n_k). \textup{ \: \: \: \: \: \: \: \: \: \: \: \: \: \: \: \: \: \: }\]
\[\: \: \: \: \: \: \: 
    {dup(n_k)}_{1 < n_k \leq n^2} = 
\begin{cases}
    0,& \text{if  m = 0}\\
    m-1,              & \text{otherwise}
\end{cases}
\: \: \: \: \: \: \:\: \: \:\: \: \]
where $m$ is the number of distinct prime factors less than or equal to $n$, of the odd integer $n_k$.
\[\textup{\: \: \: \: \: \: \: \:\: \: \: \: \: \: \: } B_i =  \frac{(n^2 - x_i) mod(2x_i) + x_i}{2x_i} .\textup{ \: \: \: \: \: \: \: \: \: \: \: \: \: \: \: \: \: \: \: \: \: \: \: }\]

\textbullet \: Equation $(5.2)$ from the prime number Theorem gives an approximation of $\pi(x)$, while equation $(5.3)$ is an "exact formula" of $\pi(n^2)$. \\

\textbullet \: 
Equation $(5.3)$ can be useful to introduce the prime counting function without going much deeper into the Riemann Hypothesis. \\

\textbullet \: The proof of equation $(5.3)$ is elementary, it follows the counting process we used to obtain the equation. \\

\textbullet \: Equation $(5.3)$ yields much deeper insights into the distribution of the prime numbers. But to have an exact value of $\pi(n^2)$, one may have to count by hand the $dup$ function. This may be adequate for small values of $n$ as illustrated in examples $10$ and $11$. But large values of $n$ would challenge us to develop and improve an approximation of the $dup$ function. \\

\textbullet \: "The precise behavior of $\pi(x) - \textup{Li}(x)$ in equation $(5.2)$ depends on the location of the zeros of the Riemann zeta function, and cannot be determined until we have more precise information about them than we have now."\cite{barkley1962approximate} While the precision of $\pi(n^2)$ in equation $(5.3)$ (for large values of $n$) depends on the approximation of $dup = \sum_{1 < k \leq n^2} dup(k)$ that is related to the summatory function of $\omega(k)$ given by $\sum_{2}^{n} \omega(k) = n\ln \ln n + B_1 n +  \mathcal{O}\left(\frac{n}{\ln n} \right).$ Where $B_1$ is a Mertens constant \cite{ram2010}.\\

\textbullet \: Equation $(5.3)$ has its limitations. Computing $\pi(n^2)$ requires that we know the primes not exceeding $ n$ to calculate $\sum_{3 \leq p_i \leq n} \frac{1}{p_i}$ and $\sum_{3 \leq p_i \leq n} B_i$. But the presence of $\sum_{3 \leq p_i \leq n} \frac{1}{p_i}$ in the formula of $\pi(n^2)$ can also be an asset if it is used as a stepping stone to the Riemann's Hypothesis. Moreover, equation $(5.3)$ is in terms of $dup = \sum_{1 < k \leq n^2} dup(k)$ that we don't know its exact formula. But due to Hardy, Ramanujan \cite{hardy2000normal} and Erdos \cite{erdos1935normal}, we have good information on $\omega(n)$ that can be useful to approximate the dup function.\\

The number of distinct prime factors of a given integer $n$, is to equation $(5.3)$ what  the location of the zeros in the Riemann zeta function is to the prime number Theorem. It opens new windows into more accurate and precise estimations of the prime counting function and refines our understanding of the distribution of prime numbers. 
According to some experts, the difficulty of the Riemann's Hypothesis is as follows: \\
"There is no approach currently known to understand the distribution of prime numbers well enough to establish the desired approximation, other than by studying the Riemann zeta function and its zeros." \cite{mathexchange} . Equation $(5.3)$ can be an alternative approach, and it can bring new perspectives in future investigations of the Riemann's Hypothesis. \\
\bibliographystyle{plain}
\bibliography{bibfile2}

\end{document}